\documentclass{article}

\usepackage[T1]{fontenc}
\usepackage[utf8]{inputenc}

\usepackage{lmodern}

\usepackage{latexsym}
\usepackage{amsmath,amssymb,amsthm,amstext,amsfonts,pict2e,amscd}
\usepackage{color}

\theoremstyle{plain}
\newtheorem{thm}{Theorem}[section]
\newtheorem{lem}[thm]{Lemma}
\newtheorem{pro}[thm]{Proposition}
\newtheorem{cor}[thm]{Corollary}

\theoremstyle{definition}
\newtheorem{defn}[thm]{Definition}
\newtheorem{rem}[thm]{Remark}

\newcommand{\pr} {\mathrm{pr}}

\begin{document}

\title{
       Partial cubes with pre-hull number at most $1$}
\author{Norbert Polat\\
        I.A.E., Universit\'{e} Jean Moulin (Lyon 3)\\
         6 cours Albert Thomas\\
         69355 Lyon Cedex 08, France\\
         \texttt{norbert.polat@univ-lyon3.fr}}
\date{}		
\maketitle

\begin{abstract}
We prove that a connected bipartite graph $G$ is a partial cube if and only if the set of attaching points of any copoint of $G$ is convex.  A consequence of this result is that any connected bipartite graph with pre-hull number at most $1$ is a partial cube.  We show that the class of partial cubes with pre-hull number at most $1$ is closed under gated subgraphs, gated amalgams and cartesian products.\\

\emph{Keywords:} Geodesic convexity; Copoint; Attaching point; Pre-hull number;  Bipartite graph; Partial cube; Median graph, Netlike partial cube.

\end{abstract}

\section{Introduction}\label{S:introd.}

The (geodesic) pre-hull number $ph(G)$ of a graph $G$ is a parameter which measures the intrinsic non-convexity of $V(G)$ in terms of the number of iterations of the pre-hull operator associated with the interval operator $I_G$ which are necessary, in the worst case, to reach the canonical minimal convex extension of copoints of $V(G)$ when they are extended by the adjunction of an attaching point.  In~\cite{PS09}, where this concept was introduced, the question whether any connected bipartite graph with pre-hull number at most $1$ is a partial cube was considered, but only partial results were obtained \cite[Sections 6 and 7]{PS09}.  Note that a connected bipartite graph with pre-hull number greater than $1$ may or may not be a partial cube.  The first part of the present paper deals with the research of a definitive answer to this question.

In~\cite{PS09} we proved that, for any copoint $K$ of a partial cube $G$, the set $\mathrm{Att}(K)$ of all attaching points of $K$ is convex (Att-convexity of $G$).  In the first part of this paper (Section~\ref{S:charact.}), we show that Att-convexity is a necessary and sufficient condition for a (finite or infinite) connected bipartite graph to be a partial cube.  The affirmative answer to the above question follows immediately: any connected bipartite graph with pre-hull number at most $1$ is a partial cube.

The class of partial cubes with pre-hull number at most $1$ contains most of the mainly studied partial cubes such as: median graphs, cellular bipartite graphs, benzenoid graphs and netlike partial cubes.  We show that this class is closed under gated subgraphs (but not convex ones), gated amalgams and cartesian products.

\section{Preliminaries}\label{S:prelim.}

 \subsection{Graphs}\label{SS:gr.}
 
The graphs we consider are undirected, without loops or multiple
edges, and may be finite or infinite.  If $x \in V(G)$, the set
$N_{G}(x) := \{y \in V(G) : xy \in E(G)\}$ is the \emph{neighborhood}
of $x$ in $G$.  For a set $S$ of vertices of a graph $G$ we put
$N_{G}(S) := \bigcup_{x \in S}N_{G}(x)-S$, and we denote by
$\partial_{G}(S)$ the \emph{edge-boundary} of $S$ in $G$, that is the
set of all edges of $G$ having exactly one end-vertex in $S$.
Moreover, $G[S]$ is the subgraph of $G$ induced by $S$, and $G-S:=
G[V(G)-S]$.

\emph{Paths} are considered as subgraphs rather than as sequences of
vertices.  Thus an $(x,y)$-path is also a $(y,x)$-path.  If $u$ and
$v$ are two vertices of a path $P$, then we denote by $P[u,v]$ the
segment of $P$ whose end-vertices are $u$ and $v$.

Let $G$ be a connected graph.  The usual \emph{distance} between two
vertices $x$ and $y$, that is, the length of any
\emph{$(x,y)$-geodesic} (= shortest $(x,y)$-path) in $G$, is denoted
by $d_{G}(x,y)$.  A connected subgraph $H$ of $ G $
is \emph{isometric} in $G$ if $d_{H}(x,y) = d_{G}(x,y)$ for all
vertices $x$ and $y$ of $H$.  The \emph{(geodesic) interval}
$I_{G}(x,y)$ between two vertices $x$ and $y$ of $G$ consists of the
vertices of all $(x,y)$-geodesics in $G$.

\subsection{Convexities}\label{SS:conv.}

A \emph{convexity} on a set $X$ is an algebraic closure system
$\mathcal{C}$ on $X$.  The elements of $\mathcal{C}$ are the
\emph{convex sets} and the pair $(X,\mathcal{C})$ is called a
\emph{convex structure}.  See van de Vel~\cite{V93} for a detailed
study of abstract convex structures.  Several kinds of graph
convexities, that is, convexities on the vertex set of a graph $G$,
have already been investigated.  We will principally work with the
\emph{geodesic convexity}, that is, the convexity on $V(G)$ which is
induced by the \emph{geodesic interval operator} $I_{G}$.  In this convexity,
a subset $C$ of $V(G)$ is convex provided it contains the geodesic
interval $I_{G}(x,y)$ for all $x, y \in C$.  The \emph{convex hull}
$co_{G}(A)$ of a subset $A$ of $V(G)$ is the smallest convex
set which contains $A$.  The convex hull of a finite set is called a
\emph{polytope}.  A subset $H$ of $V(G)$ is a \emph{half-space} if $H$
and $V(G)-H$ are convex.

A \emph{copoint} at a point $x \in X$ is a
convex set $C$ which is maximal with respect to the property that $x
\notin C$; $x$ is an \emph{attaching point} of $K$.  Note that
$co_G(K \cup \{x \}) = co_G(K \cup
\{y \})$ for any two attaching points $x, y$ of $K$.  We denote by $\textnormal{Att}(K)$ the set of all attaching
points of $K$, i.e.,
\begin{equation*}\label{E:Att}
\mathrm{Att}(K) := co_G(K \cup \{x \})-K.
\end{equation*}

We denote by $\mathcal{I}_{G}$ the \emph{pre-hull operator} of
the geodesic convex structure of $G$, i.e. the self-map of
$\mathcal{P}(V(G))$ such that $\mathcal{I}_{G}(A) := \bigcup_{x, y \in
A}I_{G}(x,y)$ for each $A \subseteq V(G)$.  The convex hull of a set
$A \subseteq V(G)$ is then $co_{G}(A) = \bigcup_{n \in
\mathbb{N}}\mathcal{I}_{G}^{n}(A)$.  Furthermore we will say that a
subgraph of a graph $G$ is \emph{convex} if its vertex set is convex,
and by the \emph{convex hull} $co_{G}(H)$ of a subgraph $H$ of $G$ we
will mean the smallest convex subgraph of $G$ containing
$H$ as a subgraph, that is, $$co_{G}(H) := G[co_{G}(V(H))].$$

\subsection{Bipartite graphs and partial cubes}\label{SS:part.cubes}

All graphs considered here are connected.

For an edge $ab$ of a graph $G$, let
\begin{align*}
	W_{ab}^{G}& := \{x \in V(G): d_{G}(a,x) < d_{G}(b,x) \},\\
	U_{ab}^{G}& := \{x \in W_{ab}: x \text{\;has a neighbor in}\; W_{ba}\}.
\end{align*}

If no confusion is likely, we will simply denote $W_{ab}^{G}$ and
$U_{ab}^{G}$ by $W_{ab}$ and $U_{ab}$, respectively.  Note that the
sets $W_{ab}$ and $W_{ba}$ are disjoint and that $V(G) = W_{ab} \cup
W_{ab}$ if $G$ is bipartite and connected.

Two edges $xy$ and $uv$ are in the Djokovi\'{c}-Winkler
relation $\Theta$ if

$$d_{G}(x,u)+d_{G}(y,v) \neq d_{G}(x,v)+d_{G}(y,u).$$

If $G$ is bipartite, the edges $xy$ and $uv$ are in
relation $\Theta$ if and only if $d_{G}(x,u) = d_{G}(y,v)$ and
$d_{G}(x,v) = d_{G}(y,u)$.  The relation $\Theta$ is clearly reflexive 
and symmetric.

\begin{lem}\label{L:bip.gr./convex}
Let $C$ be a convex set of a bipartite graph $G$.  Then $C \subseteq W_{ab}$ for any edge $ab \in \partial_{G}(C)$ with $a \in C$.
\end{lem}

\begin{proof}
Let $x \in C$ and $ab \in \partial_{G}(C)$ with $a \in C$.  Suppose that $x \notin W_{ab}$.  Then $b \in I_{G}(x,a)$, and thus $b \in C$ by the convexity of $C$, contrary to the fact that $ab \in \partial_{G}(C)$.
\end{proof}

\begin{rem}\label{R:Theta/notation}
If $G$ is bipartite, then, by~\cite[Lemma 11.2]{HIK11}, \emph{the notation can be chosen so that the edges $xy$ and $uv$ are in
relation $\Theta$ if and only if $$d_{G}(x,u) = d_{G}(y,v) = d_{G}(x,v)-1 = d_{G}(y,u)-1,$$ or equivalently if and only if $$y \in I_G(x,v) \text{\; and}\; x \in I_G(y,u).$$}
From now on, we will always use this way of defining the relation $\Theta$.  Note that, in this way, the edges $xy$ and $yx$ are not in relation $\Theta$ because $y \notin I_G(x,x)$ and $x \notin I_G(y,y)$.  In other word, each time the relation $\Theta$ is used, the notation of an edge induces an orientation of this edge.
\end{rem}

We recall the main characterizations of partial cubes, that is
of isometric subgraphs of hypercubes (see~\cite{HIK11}).  Partial cubes are particular
connected bipartite graphs.

\begin{thm}\label{T:Djokovic-Winkler}\textnormal{(Djokovi\'{c}~\cite[Theorem
1]{D73} and Winkler~\cite{W84})} A connected bipartite graph $G$ is a
partial cube if and only if it has one of the following properties:

\textnormal{(i)}\; For every edge $ab$ of $G$, the sets
$W_{ab}$ and $W_{ba}$ are convex.

\textnormal{(ii)}\;  The relation $\Theta$ is transitive.
\end{thm}

It follows in particular that \emph{the half-spaces of a partial cube
$G$ are the sets $W_{ab}$, $ab \in E(G)$}.  Furthermore we can easily 
prove that the copoints of a partial cube are its half-spaces. The following technical lemma will be used later.

\begin{lem}\label{L:prop.isom subgr.}
Let $G$ be a partial cube, $F$ an isometric subgraph of $G$, and $ab$ an edge of $F$.  Then

\textbullet\; $W_{ab}^{F} = W_{ab}^{G} \cap V(F)$\; and\; $W_{ba}^{F} = W_{ba}^{G} \cap V(F)$

\textbullet\; $U_{ab}^{F} \subseteq U_{ab}^{G} \cap V(F)$.

If moreover $F$ is convex in $G$, then

\textbullet\; $U_{ab}^{F} = U_{ab}^{G} \cap V(F)$.
\end{lem}

\begin{proof}
The first assertions are immediate consequences of the definitions of $W_{ab}$ and $U_{ab}$, and of the fact that $F$ is isometric in $G$.  Assume now that $F$ is convex in $G$.  Let $x \in U_{ab}^{G} \cap V(F)$, and let $y$ be the neighbor of $x$ in $U_{ba}^{G}$.  Then $y \in I_{G}(x,b) = I_{F}(x,b)$ since $F$ is convex.  Hence $x \in U_{ab}^{F}$.  Therefore $U_{ab}^{F} \supseteq U_{ab}^{G} \cap V(F)$, and we are done by the above converse inclusion.
\end{proof}

In the following lemma we list some well-known properties of partial 
cubes.

\begin{lem}\label{L:gen.propert.}
Let $G$ be a partial cube.  We have the following properties:

\textnormal{(i)}\; Each interval of $G$ is finite and convex.

\textnormal{(ii)}\; Each polytope of $G$ is finite.

\textnormal{(iii)}\; Let $x,y$ be two vertices of $G$, $P$ an
$(x,y)$-geodesic and $W$ an $(x,y)$-path of $G$.  Then each edge of
$P$ is in relation $\Theta$ with some edge of $W$.

\textnormal{(iv)}\;  A path $P$ in $G$ is a geodesic if and only if no two distinct edges of $P$ are $\Theta$-equivalent.
\end{lem}

\section{A characterization of partial cubes}\label{S:charact.}

\begin{defn}\label{D:Att-convex}
A bipartite graph $G$ is said to be \emph{Att-convex} if for each copoint $K$ of $G$,
the set $\mathrm{Att}(K)$ is convex.
\end{defn}

We now state the main result of this section.

\begin{thm}\label{T:bip.gr./Att-convex}
A connected bipartite graph $G$ is a 
partial cube if and only if it is Att-convex.
\end{thm}

To prove this theorem we will need several properties of expansions of
a graph, a concept which was introduced by Mulder~\cite{M78} to
characterize median graphs and which was later generalized by
Chepoi~\cite{Che88}.

The \emph{cartesian product} $G_{1} \Box G_{2}$ of two graphs $G_{1}$ 
and $G_{2}$ is the graph with vertex set $V(G_{1}) \times V(G_{2})$,
two vertices being adjacent if they have one coordinate adjacent and 
the other equal.

\begin{defn}\label{D:proper cover} 
A pair $(V_{0},V_{1})$ of sets of vertices of a graph $G$ is called a
\emph{proper cover} of $G$ if it satisfies the following conditions:

\textbullet\;  $V_{0} \cap V_{1} \neq \emptyset$ and $V_{0} \cup V_{1} = 
V(G)$;

\textbullet\;  there is no edge between a vertex in $V_{0}-V_{1}$ and 
a vertex in $V_{1}-V_{0}$;

\textbullet\;  $G[V_{0}]$ and $G[V_{1}]$ are isometric subgraphs of 
$G$.
\end{defn}

\begin{defn}\label{D:expansion}
An \emph{expansion} of a graph $G$ with respect to a proper cover 
$(V_{0},V_{1})$ of $G$ is the subgraph  of $G \Box K_{2}$ induced by
the vertex set $(V_{0} \times \{0 \}) \cup (V_{1} \times \{1 \})$ (where
$\{0,1 \}$ is the vertex set of $K_{2}$).
\end{defn}

An expansion of a partial cube is a partial cube (see~\cite{O08}).  If $G'$ is an
expansion of a graph $G$, then we say that $G$ is
a \emph{$\Theta$-contraction} of $G'$, because, as we can easily
see, $G$ is obtained from $G'$ by contracting each element of
some $\Theta$-class of edges of $G'$.  More precisely the natural surjection of $G'$ onto $G$ is a contraction, that is, an application which maps any two adjacent vertices to adjacent vertices or to a single vertex.  A $\Theta$-contraction of a partial cube is a partial cube as well (see~\cite{O08}).

In Lemmas~\ref{L:interv.}--\ref{L:cop.G0/cop.G1}, $G$ will be a
connected bipartite graph and $G'$ an expansion of $G$ with
respect to a proper cover $(V_{0},V_{1})$ of $G$.  The following
notation will be used.

\textbullet\; For $i = 0, 1$ denote by $\psi_{i} : V_{i} \to V(G')$
the natural injection $\psi_{i} : x \mapsto (x,i)$, $x \in V_{i}$, and
let $V'_{i} := \psi_{i}(V_{i})$.  Note that $V'_{0}$ and $V'_{1}$ are
complementary half-spaces of $G'$.  It follows in particular that
these sets are copoints of $G'$.

\textbullet\; For any vertex $x$ of $G$ (resp.  $G'$),
denote by $i(x)$ an element of $\{0,1 \}$ such that $x$ belongs to
$V_{i(x)}$ (resp.  $V'_{i(x)}$).  If $x \in V(G')$ and also if $x \in 
V(G) - (V_{0} \cap V_{1})$, then $i(x)$ is unique; if $x \in V_{0} \cap
V_{1}$ it may be $0$ or $1$.

\textbullet\; For $A \subseteq V(G)$ put $$\psi(A) := \psi_{0}(A \cap V_{0})
\cup \psi_{1}(A \cap V_{1}).$$  Note that in the opposite direction we
have that for any $A' \subseteq V(G')$, $$\textnormal{pr}(A') =
\psi_{0}^{-1}(A' \cap V'_{0}) \cup
\psi_{1}^{-1}(A' \cap V'_{1}),$$ where $\textnormal{pr} : G \Box K_{2}
\to G$ is the projection $(x,i) \mapsto x$.

The following lemma is a restatement with more precisions of
\cite[Lemma 4.5]{P05-1}.

\begin{lem}\label{L:interv.}
Let $G$ be a connected bipartite graph and $G'$ an expansion of
$G$ with respect to a proper cover $(V_{0},V_{1})$ of
$G$, and let $P = \langle x_{0},\ldots,x_{n}\rangle$ be a path in 
$G$.  We have the following properties:

\textnormal{(i)}\; If $x_{0}, x_{n} \in V_{i}$ for some $i = 0$ or $1$, then:

\textbullet\;  if $P$ is a geodesic in $G$, then there exists an $(x_{0},x_{n})$-geodesic $R$ in $G[V_i]$ such that $V(P) \cap V_i \subseteq V(R)$;

\textbullet\;  $P$ is a geodesic in $G[V_i]$ if and only if $P' = \langle \psi_{i}(x_{0}),\ldots,\psi_{i}(x_{n}) \rangle$ is
a geodesic in $G'$;  

\textbullet\;  $d_{G'}(\psi_{i}(x_{0}),\psi_{i}(x_{n}))
= d_{G}(x_{0},x_{n})$;

\textbullet\;  $I_{G'}(\psi_{i}(x_{0}),\psi_{i}(x_{n})) =
\psi_{i}(I_{G[V_i]}(x_{0},x_{n})) \subseteq \psi(I_{G}(x_{0},x_{n}))$.

\textnormal{(ii)}\; If $x_0 \in V_i$ and $x_1 \in V_{1-i}$ for some $i = 0$ or $1$, then:

\textbullet\;  if there exists $p$ such that $x_0,\ldots,x_p \in V_i$ and $x_p,\ldots,x_n \in V_{1-i}$, then $P$ is a geodesic in $G$ if and only if the path $$P' = \langle
\psi_{i}(x_{0}),\ldots,\psi_{i}(x_{p}),\psi_{1-i}(x_{p}),\ldots,\psi_{1-i}(x_{n})
\rangle$$ is a geodesic in $G'$;

\textbullet\;  $d_{G'}(\psi_{i}(x_{0}),\psi_{1-i}(x_{n}))
= d_{G}(x_{0},x_{n}) + 1$;

\textbullet\;  $I_{G'}(\psi_{i}(x_{0}),\psi_{1-i}(x_{n})) =
\psi(I_{G}(x_{0},x_{n}))$.
\end{lem}

From Corollary~\ref{C:conv.G0/conv.G1} to Lemma~\ref{L:cop.G0/cop.G1}, $G$ will be a connected bipartite graph and $G'$ an expansion of $G$ with
respect to a proper cover $(V_{0},V_{1})$ of $G$.  The following result is an immediate consequence of Lemma~\ref{L:interv.}.

\begin{cor}\label{C:conv.G0/conv.G1}
Let $K$ be a convex set of $G$.  Then $\psi(K)$ is a convex set of 
$G'$.
\end{cor}

\begin{lem}\label{L:conv.hull}
$\pr(co_{G'}(S)) \subseteq co_G(\pr(S))$ for any $S \subseteq V(G')$.
\end{lem}

\begin{proof}
We have $$S \subseteq \psi(\pr(S)) \subseteq \psi(co_G(\pr(S))).$$By Corollary~\ref{C:conv.G0/conv.G1}, $\psi(co_G(\pr(S)))$ is convex in $G'$.  Hence $co_{G'}(S) \subseteq \psi(co_G(\pr(S)))$.  Therefore $\pr(co_{G'}(S)) \subseteq co_G(\pr(S))$.
\end{proof}

\begin{lem}\label{L:conv.G1/conv.G0}
Let $K'$ be a convex set of $G'$ which meets both $V'_{0}$ and
$V'_{1}$.  Then $K := \textnormal{pr}(K')$ is a 
convex set of $G$.
\end{lem}

\begin{proof}
Let $u, v \in K$.  If $i(u) \neq i(v)$, then $I_{G}(u,v) =
\textnormal{pr}(I_{G'}(u',v'))$ by
Lemma~\ref{L:interv.}, and hence $I_{G}(u,v) \subseteq K$.

Now assume that $i(u) = i(v)$, say $i(u) = i(v) = 0$.  Let $P =
\langle x_{0},\ldots,x_{n} \rangle$ be a $(u,v)$-geodesic in $G$ with
$x_{0} = u$ and $x_{n} = v$.  In general, not all of $P$ is contained
in $G[V_{0}]$.  Let $0 = i_{0} < i_{1} < \ldots < i_{2p + 1} = n$ be
subscripts such that the segments $P[x_{i_{0}},x_{i_{1}}],
P[x_{i_{1}},x_{i_{2}}],\ldots, P[x_{i_{2p}},x_{i_{2p+1}}]$ are
alternatively contained in $G[V_{0}]$ and $G[V_{1}]$.  Thus
$x_{i_{1}}, \ldots, x_{i_{2p}} \in V_{0} \cap V_{1}$.  Since
$G[V_{0}]$ is isometric in $G$ there is an
$(x_{i_{2h-1}},x_{i_{2h}})$-geodesic $P_{h}$ in $G[V_{0}]$, $h =
1,\ldots, p$.  Replacing each $(x_{i_{2h-1}},x_{i_{2h}})$-segment of
$P$ by the corresponding $P_{h}$ one obtains a new $(u,v)$-geodesic
$P_{0}$ with $V(P_{0}) \subseteq V_{0}$.  Hence $\psi_{0}(P_{0})$ is a
$(u',v')$-geodesic in $G'$, and therefore $V(P_{0}) \subseteq K$.

It follows in particular that $\psi_{0}(x_{i_{k}}) \in K' \cap V'_{0}$,
$k = 1, \ldots, 2p$.  By hypothesis there exists a vertex $w \in K'
\cap V'_{1}$.  From the construction of $G'$ it then follows that
$y_{k} := \psi_{1}(x_{i_{k}}) \in I_{G'}(\psi_{0}(x_{i_{k}}),w)$, and
hence $y_{k} \in K'$.  Since $G[V_{1}]$ is an isometric subgraph of
$G$ we deduce that $\psi_{1}(P[x_{i_{2k-1}},x_{i_{2k}}])$ is a
$(y_{2k-1},y_{2k})$-geodesic.  Hence $V(P[x_{i_{2k-1}},x_{i_{2k}}])
\subseteq K$, and therefore $V(P) \subseteq K$.
\end{proof}

\begin{lem}\label{L:copoint-G1/copoint-G0}
If $K'$ is a copoint of $G'$ which meets both $V'_{0}$ and $V'_{1}$,
then $K := \textnormal{pr}(K')$ is a copoint of $G$ such that
$\mathrm{Att}(K) = \textnormal{pr}(\mathrm{Att'}(K'))$ (where
$\mathrm{Att}$ and $\mathrm{Att'}$ denote the sets of
attaching points in $G$ and $G'$, respectively).
\end{lem}

\begin{proof}      
Let $u \in \mathrm{Att'}(K')$ and abbreviate $i(u)$ by $i$.  Thus
$u \in V'_{i}$.  By Lemma~\ref{L:conv.G1/conv.G0}, $K$ is a convex set
of $G$.  Moreover $x := \textnormal{pr}(u) \notin K$.  Suppose that
$K$ is not a copoint at $x$.  Then $G$ contains a convex set $K_{0}$
with $x \notin K_{0}$ and $K \subseteq K_{0}$.  By
Corollary~\ref{C:conv.G0/conv.G1}, $\psi(K_{0})$ is a convex set of
$G'$ which strictly contains $K'$.  Hence $u \in \psi(K_{0})$ because
$K'$ is a copoint at $u$, contrary to the fact that $x \notin K_{0}$.
Consequently $K$ is a copoint at $x$.

It follows that $\textnormal{pr}(\mathrm{Att'}(K')) \subseteq
\mathrm{Att}(K)$.  On the other hand, by
Lemma~\ref{L:conv.G1/conv.G0}, $\textnormal{pr}(co_{G'}(\{u \} \cup
K'))$ is a convex set of $G$ containing $\{x \} \cup K$.  Hence
$\mathrm{Att}(K) \subseteq \textnormal{pr}(\mathrm{Att'}(K'))$.
\end{proof}

Going from $G$ to $G'$ we have:

\begin{lem}\label{L:cop.G0/cop.G1}
If $K$ is a copoint of $G$ which meets $V_{0} \cap V_{1}$, then
$K' := \psi(K)$ is a copoint of $G'$ such that $\mathrm{Att'}(K') =
\psi(\mathrm{Att}(K))$.
\end{lem}

\begin{proof}
Let $x \in \mathrm{Att}(K)$.  By
Corollary~\ref{C:conv.G0/conv.G1}, $\psi(K)$ is a convex set of $G'$
such that $\psi_{i(x)}(x) \notin \psi(K)$.  Let $K'$ be a copoint of
$G'$ at $\psi_{i(x)}(x)$ which contains $\psi(K)$.  Then $K' \cap
V_{i}^{1} \neq \emptyset$ for $i = 0, 1$.  By
Lemma~\ref{L:copoint-G1/copoint-G0}, $\textnormal{pr}(K')$ is a
copoint of $G$ at $x$ which contains $K$, and thus is equal to $K$.
Hence $K' = \psi(K)$.  Now $\mathrm{Att'}(\psi(K)) \subseteq
\psi(\mathrm{Att}(K))$ by Lemma~\ref{L:copoint-G1/copoint-G0}, and
moreover $(u,0), (u,1) \in \psi(K)$ for each $u \in K \cap V_{0} \cap
V_{1}$.  It follows that $\mathrm{Att'}(\psi(K)) =
\psi(\mathrm{Att}(K))$.
\end{proof}

\begin{lem}\label{L:G'Att-conV;=>GAtt-conv.}
$G$ is Att-convex if so is $G'$.
\end{lem}

\begin{proof}
Assume that $G'$ is Att-convex.  Let $K$ be a copoint of $G$.  We will show that $\mathrm{Att}(K)$ is 
convex.  We distinguish two cases.

\emph{Case 1.}\;  $K \cap V_{i} = \emptyset$ for some $i \in \{0,1 \}$.

Say $i = 0$.  Hence $K \subseteq V_{1} - V_{0}$.  Then
$K = \psi(K)$ is convex in $G'$ by Corollary~\ref{C:conv.G0/conv.G1}.  Let $A := \psi(\mathrm{Att}(K))$.  Then $K \cup A = \psi(K \cup \mathrm{Att}(K))$, and thus $K \cup A$ is convex in $G'$ by Corollary~\ref{C:conv.G0/conv.G1} since $K \cup \mathrm{Att}(K)$ is convex in $G$.

Let $u \in A \cap V'_1$, $u' = \psi_1(u)$, and let $K'$ be a copoint at $u'$ in $G'$ containing $K$ or equal to $K$.  Suppose that $K' \cap V'_{0} \neq \emptyset$.  Then, by Lemma~\ref{L:copoint-G1/copoint-G0}, $\pr(K')$ is a copoint at $u$ in $G$ with $K \subset K'$, contrary to the fact that $K$ is a copoint at $u$.  Therefore
\begin{equation}
\label{E:K'/v'1}
K' \subseteq V'_1.
\end{equation}

Suppose that $\mathrm{Att'}(K') \cap V'_0 \neq \emptyset$.  Then, because $K' \subseteq V'_1$ by (\ref{E:K'/v'1}) and since $V'_1$ is convex, there exists a vertex $x_0 \in \mathrm{Att'}(K') \cap V'_0 \cap N_{G'}(K')$.  Let $x_1$ be the neighbor of $x_0$ in $V'_1$.  Then $x_1 \in K'$ and $x_1 \in I_{G'}(u',x_0)$ by Lemma~\ref{L:bip.gr./convex}, contrary to the fact that $\mathrm{Att'}(K')$ is convex since $G'$ is Att-convex by assumption.  Therefore
\begin{equation}
\label{E:Att'(K')/v'1}
\mathrm{Att'}(K') \subseteq V'_1.
\end{equation}.

Suppose that $A \cap K' \neq \emptyset$, and let $x \in \mathrm{Att}(K)$ be such that $\psi_1(x) \in K'$.  Because $K \cup \mathrm{Att}(K) = co_{G}(K \cup \{x\})$ since $K$ is a copoint at $x$, it follows that $u \in co_{G}(K \cup \{x\})$.  Hence $u' \in co_{G'}(K \cup \{\psi_1(x)\}) \subseteq K'$, contrary to the facts that $K'$ is a copoint at $u'$.  Therefore 
\begin{equation}
\label{E:A/Att'(K')}
A \cap V'_1 \subseteq \mathrm{Att'}(K').
\end{equation}

 We distinguish two subcases.

\emph{Subcase 1.1.}\; $(K \cup \mathrm{Att}(K)) \cap V_0 = \emptyset$.

Then $\mathrm{Att}(K) = \pr(A) = A \subseteq \mathrm{Att'}(K')$  by (\ref{E:A/Att'(K')}).  Hence $co_G(\mathrm{Att}(K)) \cap K = \emptyset$ since $\mathrm{Att'}(K')$ is convex and disjoint from $K'$, and thus from $K$.  Therefore $\mathrm{Att}(K)$ is convex since so is $K \cup \mathrm{Att}(K)$.\\

\emph{Subcase 1.2.}\; $(K \cup \mathrm{Att}(K)) \cap V_0 \neq \emptyset$.

Then $\mathrm{Att}(K) \cap V_0 \cap V_1 \neq \emptyset$, and thus $A \cap V'_i \neq \emptyset$ for $i = 0, 1$.  The set $A \cap V'_0$, which is equal to $(K \cup A) \cap V'_0$, and the set $(K \cup A) \cap V'_1$ are convex since so are the sets $K \cup A$,\; $V'_0$ and $V'_1$.  By (\ref{E:A/Att'(K')}) and the fact that $\mathrm{Att'}(K')$ is convex since $G'$ is Att-convex by assumption, we infer that $co_{G'}(A \cap V'_1) \subseteq \mathrm{Att'}(K')$, and thus $co_{G'}(A \cap V'_1) \cap K = \emptyset$ since $K \subseteq K'$.  Because $(K \cup A) \cap V'_1 = K \cup (A \cap V'_1)$ is convex, it follows that $A \cap V'_1$ is also convex.  Hence $A$, which is equal to the union of the two convex sets $A \cap V'_0$ and $A \cap V'_1$, is convex by Lemma~\ref{L:interv.}.  Therefore $\mathrm{Att}(K) = \pr(A)$ is convex by Lemma~\ref{L:conv.G1/conv.G0} since $A \cap V'_i \neq \emptyset$ for $i = 0, 1$.\\

\emph{Case 2.}\; $K \cap V_{0} \cap V_{1} \neq \emptyset$.

By Lemma~\ref{L:cop.G0/cop.G1}, $K' := \psi(K)$ is a copoint of $G'$
such that $\mathrm{Att'}(K') = \psi(\mathrm{Att}(K))$.  The set 
$\mathrm{Att'}(K')$ is convex because $G'$ is Att-convex by assumption.
Furthermore $K' \cap V'_{1} \neq \emptyset$ for $i = 0, 1$.  Hence, by
Lemma~\ref{L:copoint-G1/copoint-G0}, $\textnormal{pr}(K')$ is a
copoint of $G$ such that $\mathrm{Att}(\textnormal{pr}(K')) =
\textnormal{pr}(\mathrm{Att'}(K'))$.  Because $\mathrm{Att}(K)
\subseteq \textnormal{pr}(\mathrm{Att'}(K'))$ and since $K$ is a
copoint, it follows that $K = \textnormal{pr}(K')$ and
$\mathrm{Att}(K) = \textnormal{pr}(\mathrm{Att'}(K'))$.

If $\mathrm{Att'}(K') \cap V'_{i} \neq \emptyset$ for $i = 0, 1$,
then $\mathrm{Att}(K)$ is convex by Lemma~\ref{L:conv.G1/conv.G0} since $\mathrm{Att'}(K')$ is convex.  Suppose that $\mathrm{Att'}(K') \subseteq V'_{i}-V'_{1-i}$ for some $i = 0$ or $1$.  Then $\mathrm{Att}(K) =
\psi_{i}^{-1}(\mathrm{Att'}(K')) \subseteq V_{i}-V_{1-i}$.  It follows that
$\mathrm{Att}(K) = \mathrm{Att'}(K')$.  Therefore $\mathrm{Att}(K)$ is convex since so is $\mathrm{Att'}(K')$.\\

Consequently $G$ is Att-convex.
\end{proof}

\begin{lem}\label{L:inf.p.c.}
A bipartite graph $G$ is a partial cube if and only if every polytope of $G$ induces a partial cube.
\end{lem}

\begin{proof}
We only have to prove the sufficiency.  Let $ab$ be an edge of $G$,
and let $cd$ and $ef$ be two other edges of $G$ such that each of them
is in relation $\Theta$ with $ab$.  Then the polytope $A := co_G(a,b,c,d,e,f)$ induces a partial cube $F$ by hypothesis.  Because $F$ is a convex subgraph of $G$, it follows that both the edges $cd$ and $ef$ are in relation $\Theta$ with $ab$ in $F$.  Because $F$ is a partial cube, we infer from Theorem~\ref{T:Djokovic-Winkler} that these edges are in relation $\Theta$ in $F$, and thus in $G$.  Consequently the relation $\Theta$ in $G$ is transitive, which proves that $G$ is a partial cube by Theorem~\ref{T:Djokovic-Winkler}.
\end{proof}

\begin{lem}\label{L:attach.pts/conv}
Any convex subgraph of an Att-convex graph is also Att-convex.
\end{lem}

\begin{proof}
Let $H$ be a convex subgraph of an Att-convex graph $G$, and let $K$ be a copoint at a vertex $x$ of $H$.  Then $K$ is convex in 
$G$, and thus it is contained in a copoint $K'$ at $x$ in $G$.  
Clearly $K = K' \cap V(H)$.  Moreover $\mathrm{Att}(K) \subseteq \mathrm{Att}(K') \cap V(H)$.  
Because $V(H)$ and $\mathrm{Att}(K')$ are convex in $G$, it follows that $\mathrm{Att}(K)$ is contained in a convex subset of $V(H)$ which does not meet $K$.  It follows that $\mathrm{Att}(K)$ is convex because so is $K \cup \mathrm{Att}(K)$.
\end{proof}

\begin{proof}[\textnormal{\textbf{Proof of
Theorem~\ref{T:bip.gr./Att-convex}}}]

We only have to prove the sufficiency because of~\cite[Theorem 6.7]{PS09} which in particular states that \emph{a connected bipartite graph $G$ is a partial cube if and only if it is Att-convex and $N_G(K) \subseteq \mathrm{Att}(K)$ for each copoint $K$ of $G$}.

\emph{Case 1.}\;  $G$ is finite.

The proof will be by induction on the order of $G$.  This is obvious if $G$ has one or two vertices since $K_{1}$ and $K_{2}$ are hypercubes..  Let $n \geq 2$.  Suppose that every connected bipartite graph whose order is at most $n$ and which is Att-convex is a partial cube.  Let $G$ be an Att-convex connected bipartite graph whose order is $n+1$.

Because $G$ is finite, there exists a copoint $K$ of $G$ which is
maximal with respect to inclusion.  Then $\mathrm{Att}(K) =
V(G)-K$, since otherwise there would exist a copoint at some vertex $x
\notin K \cup \mathrm{Att}(K)$ strictly containing $K$, contrary
to the maximality of $K$.  Because $G$ is Att-convex, $\mathrm{Att}(K)$ is
convex and thus $K$ is a half-space.  Therefore the edges in $\partial_G(K)$ are pairwise in relation $\Theta$.

Let $F$ be the graph obtained from $G$ by identifying,
for each edge between $K$ and $V(G)-K$, the endvertices of this
edge.  Clearly $G$ is an expansion of $F$.  Note that $F$ is a 
bipartite graph whose order is at most $n$, and that it is Att-convex by Lemma~\ref{L:G'Att-conV;=>GAtt-conv.}.  Hence $F$ is a partial cube by the induction hypothesis.  Therefore $G$ is also a partial cube by the properties of expansions.\\

\emph{Case 2.}\;  $G$ is infinite.

We denote by $\mathcal{C}$ the class
of all Att-convex connected bipartite graphs whose vertex set is a polytope.  Let $H \in \mathcal{C}$.  A subset $S$ of $V(H)$ such that $V(H) = co_H(S)$ is called a \emph{spanning set} of $H$.  We define:
\begin{align*}
d(S) &:= \Sigma_{x, y \in S}d_H(x,y)\\
d(H) &:= \min\{d(S): S \textnormal{\;is a finite spanning set of}\; H\}.
\end{align*}

\emph{Claim.\; Any $H \in \mathcal{C}$ is a finite partial cube.}

We first prove by induction on $d(H)$ that any $H \in \mathcal{C}$ is finite.  This is obvious if $d(H) = 0$ since $H = K_1$.  Let $n$ be a non-negative integer.  Suppose that any $H \in \mathcal{C}$ such $d(H) \leq n$ is finite.  Let $H \in \mathcal{C}$ be such that $d(H) = n+1$, and let $S$ be a finite spanning set of $H$ such that $d(S) = d(H)$.  By~\cite{V93}, $V(H)$ cannot be the union of a non-empty chain of proper convex subsets.  Hence $V(H)$ contains a maximal convex subset $K$.  Then $K$ is a copoint of any element of $V(H)-K$, i.e. $\mathrm{Att}(K) = V(H)-K$.  It follows that $K$ is a half-space since $\mathrm{Att}(K)$ is convex because $H$ is Att-convex by hypothesis.  Therefore the edges in $\partial_H(K)$ are pairwise in relation $\Theta$.

Let $F$ be the graph obtained from $H$ by identifying,
for each edge between $K$ and $V(H)-K$, the endvertices of this
edge.  Clearly $H$ is an expansion of $F$.  By Lemma~\ref{L:G'Att-conV;=>GAtt-conv.}, $F$ is Att-convex.  Let $S$ be a finite spanning set of $H$.  By Lemma~\ref{L:conv.hull} we have $$V(F) = \pr(V(H)) = \pr(co_H(S)) \subseteq co_F(\pr(S)) \subseteq V(F).$$  Hence $V(F) = co_F(\pr(S))$, i.e. $\pr(S)$ is a finite spanning set of $F$.  It follows that $F \in \mathcal{C}$.  On the other hand, because $S$ is a finite spanning set of $H$, and because $K$ is a half-space, it follows that $K$ and $V(H)-K$ have non-empty intersections with $S$.  Therefore $d(F) \leq d(\pr(S)) < d(S) = d(H) = n+1$.  Hence $d(F) \leq n$, and thus, by the induction hypothesis, $F$ is finite.  It follows that $H$, which is an expansion of $F$, is also finite.

$H$ is then a finite connected bipartite graph which is Att-convex.  
We then deduce, by Case 1 of this proof, that $H$ is a partial cube, which completes the proof of the claim.\\

Now, let $G$ be an infinite Att-convex bipartite graph.  By Lemma~\ref{L:attach.pts/conv}, each polytope of $G$ is Att-convex, and thus is a partial cube by the above claim.  Consequently $G$ is itself a partial cube by Lemma~\ref{L:inf.p.c.}.
\end{proof}

\begin{pro}\label{P:charact.p.c.}
Let $G$ be a connected bipartite graph.  The following assertions are equivalent:

\textnormal{(i)}\; $G$ is a partial cube.

\textnormal{(ii)}\; $G$ is Att-convex.

\textnormal{(iii)}\; For every convex subgraph $F$ of $G$, any maximal proper convex subset of $V(F)$ is a half-space of $F$.
\end{pro}

\begin{proof}
(i) $\Rightarrow$ (iii):\; Let $F$ be a convex subgraph of a partial cube $G$.  Then $F$ itself is a partial cube.  Let $K$ be a maximal proper convex subset of $V(F)$.  Then $K = W_{ab}$ for some edge $ab \in \partial_F(K)$ with $a \in K$.  Hence $V(F)-K = W_{ba}$, which proves that $K$ is a half-space of $F$ by Theorem~\ref{T:Djokovic-Winkler}.

(iii) $\Rightarrow$ (ii):\; Assume that $G$ satisfies (iii), and let $K$ be a copoint of $G$.  Then $X := K \cup \mathrm{Att}(K)$ is convex, and $K$ is a maximal proper convex subset of $X$.  Hence $K$ is a half-space of $G[X]$ by (iii).  Therefore $\mathrm{Att}(K)$ is convex in $G[X]$, and thus in $G$.

(ii) $\Leftrightarrow$ (i) is Theorem~\ref{T:bip.gr./Att-convex}.
\end{proof}

\section{Partial cubes with pre-hull number at most $1$}\label{S:p.c.ph.at most1}

We begin by recalling some definitions and results from~\cite{PS09}.  In that paper we introduced and studied the concept of pre-hull
number of a convexity.  We recall its definition in the particular
case of the geodesic convexity of a graph.

\begin{defn}\label{D:ph(G)}
Let $G$ be a graph.  The least non-negative integer $n$ (if it exists) such that
$co_{G}(C \cup \{x \}) = \mathcal{I}_{G}^{n}(C \cup \{x \})$ for each
vertex $x$ of $G$ and each copoint $C$ at $x$, is called the \emph{pre-hull number} of a graph $G$ and is denoted by
$ph(G)$.  If there is no such
$n$ we put $ph(G) := \infty$.
\end{defn}

\begin{pro}\label{P:ph(bipart.gr.)=0/trees}\textnormal{(Polat and Sabidussi~\cite[Corollary 3.8]{PS09})}
The pre-hull number of a connected bipartite graph $G$ is zero if and
only if $G$ is a tree.
\end{pro}

\begin{defn}\label{D:ph-stable}\textnormal{(Polat and Sabidussi~\cite[Definition 7.1]{PS09})}
Call a set $A$ of vertices of a graph $G$ \emph{ph-stable} if any two
vertices $u,v \in \mathcal{I}_{G}(A)$ lie on a geodesic joining two
vertices in $A$.
\end{defn}

The condition of Definition~\ref{D:ph-stable}, which is
symmetric in $u$ and $v$, can be replaced by the formally
\textquotedblleft one-sided\textquotedblright condition: \emph{for any
two vertices $u,v \in \mathcal{I}_{G}(A)$ there is a $w \in A$
such that $v \subseteq I_{G}(u,w)$}.

\begin{pro}\label{P:bip.gr./ph=1}\textnormal{(Polat and Sabidussi~\cite[Theorem 7.4]{PS09})}
Let $G$ be a bipartite graph.  Then $ph(G) \leq 1$ if and only if, for
every copoint $K$ of $G$, the set $\mathrm{Att}(K)$ is convex and
$N_G(K) \cap \mathrm{Att}(K)$ is ph-stable.
\end{pro}

The following result follows immediately from the above proposition.

\begin{pro}\label{P:charact.p.c.ph}\textnormal{(Polat and Sabidussi~\cite[Theorem 7.5]{PS09})}
Let $G$ be a partial cube.  Then $ph(G) \leq 1$ if and only if
$U_{ab}$ and $U_{ba}$ are ph-stable for every edge $ab$ of $G$.
\end{pro}

From Theorem~\ref{T:bip.gr./Att-convex} and Proposition~\ref{P:bip.gr./ph=1} we infer the second main result of this paper.

\begin{thm}\label{T:bip.gr.ph1=>p.c.}
Any connected bipartite graph $G$ such that $ph(G) \leq 1$ is a 
partial cube.
\end{thm}

Note that a bipartite graph whose pre-hull number is greater than $1$ may or may not be a partial cube.  For example, $2$ is the pre-hull number of both $K_{2,3}$, which is the smallest connected bipartite graph which is not a partial cube, and of the partial cube $Q_3^-$ (i.e. the $3$-cube $Q_3$ minus a vertex).  A lot of well-known partial cubes have a pre-hull number equal to $1$: median graphs, benzenoid graphs, cellular bipartite graphs and more generally netlike partial cubes.

We will now study some properties of partial cubes whose
pre-hull number is at most~$1$, with in particular the closure of the class of these graphs  under usual operations of partial cubes.

\begin{pro}\label{P:fin.conv./ph1}
Let $G$ be a partial cube such that any finite subgraph of $G$ 
is contained in a finite convex subgraph of $G$ whose pre-hull number 
is at most $1$.  Then $ph(G) \leq 1$.
\end{pro}

\begin{proof}
Let $ab \in E(G)$ and $u, v \in \mathcal{I}_{G}(U_{ab})$.  Let $P_u$ and $P_v$ be geodesics joining vertices in $U_{ab}$ on which lie $u$ and $v$, respectively.  Then $\langle a,b\rangle \cup P_u \cup P_v$ is contained in a finite convex subgraph $F$ of $G$ such that $ph(G) \leq 1$.  The set $U_{ab}^F$ is ph-stable since $ph(F) \leq 1$, and thus $u, v$ lie on an $(x,y)$-geodesic $R$ for some $x, y \in U_{ab}^F$.  Because $F$ is convex in $G$, it follows that $R$ is a geodesic in $G$, and also that $x, y \in U_{ab}$ since $U_{ab}^{F} = U_{ab} \cap V(F)$ by Lemma~\ref{L:prop.isom subgr.}.  Therefore $U_{ab}$, and analogously $U_{ba}$, are ph-stable.  Hence $ph(G) \leq 1$ by Proposition~\ref{P:charact.p.c.ph}.
\end{proof}

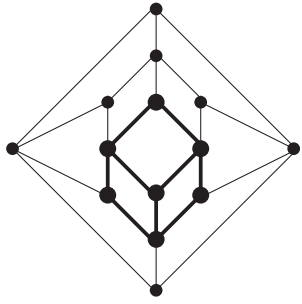
\begin{figure}[!h]
    \centering
   {\tt    \setlength{\unitlength}{0.80pt}
\begin{picture}(233,180)
\thinlines    

              \put(140,66){\line(2,1){43}}
              \put(140,111){\line(2,-1){43}}
              \put(97,66){\line(-2,1){45}}
              \put(96,111){\line(-2,-1){45}}
              \put(51,89){\line(1,-1){68}}
              \put(184,90){\line(-1,1){65}}
              \put(51,89){\line(1,1){66}}
              \put(184,87){\line(-1,-1){65}}
              \put(119,46){\line(0,-1){23}}
              \put(96,88){\line(0,1){26}}
              \put(140,90){\line(0,1){24}}
              \put(119,110){\line(0,1){26}}
              \put(119,134){\line(-1,-1){24}}
              \put(119,134){\line(1,-1){22}}
              \put(119,136){\line(0,1){20}}
              \put(119,133){\circle*{6}}
              \put(184,89){\circle*{6}}
              \put(51,89){\circle*{6}}
              \put(119,155){\circle*{6}}
              \put(119,22){\circle*{6}}
              \put(140,111){\circle*{6}}
              \put(96,111){\circle*{6}}
              \put(140,89){\circle*{8}}
              \put(96,89){\circle*{8}}
              \put(119,111){\circle*{8}}
              \put(119,68){\circle*{8}}
              \put(140,67){\circle*{8}}
              \put(96,67){\circle*{8}}
              \put(119,46){\circle*{8}}
 \linethickness{0,5mm}
              \put(140,90){\line(0,-1){23}}
              \put(96,89){\line(0,-1){22}}
              \put(119,44){\line(0,1){26}}
              \put(120,113){\line(-1,-1){24}}
              \put(119,112){\line(1,-1){22}}
              \put(97,66){\line(1,-1){22}}
              \put(141,68){\line(-1,-1){24}}
              \put(141,90){\line(-1,-1){24}}
              \put(97,88){\line(1,-1){22}}
\end{picture}}    
\caption{$M_{4,1}$ with a copy of $Q_3^-$ as a convex subgraph.}
\label{F:conv.subgr.}
\end{figure}

As was shown in \cite[Remark 8.1]{PS09}, the class of all partial cubes whose
pre-hull number is at most~$1$ is not closed under convex subgraphs.  The graph $M_{n,1}$,\; $n \geq 4$, i.e. the cube $Q_n$ from which a pair of antipodal vertices has been removed, has a pre-hull number equal to $1$.  On the other hand $M_{n,1}$ contains copies of $Q_{n-1}^{-}$ (the cube $Q_{n-1}$ with only one vertex deleted) as convex subgraphs (see Figure~\ref{F:conv.subgr.} for $n = 4$, where $Q_3^-$ is depicted by the big points and the thick lines), and $ph(Q_{n-1}) = 2$ by \cite[Theorem 5.8]{PS09}.   It was also shown in \cite[Remark 3.3]{P05-2} that $Q_3^-$ is a retract of $M_{4,1}$, which proves that the class of all partial cubes whose
pre-hull number is at most~$1$ is not closed under retracts.  However, we will see that it is closed under  gated subgraphs.

A set $A$ of vertices of a graph $G$ is said to be \emph{gated} if,
for each $x \in V(G)$, there exists a vertex $y$ (the \emph{gate} of
$x$) in $A$ such that $y \in I_{G}(x,z)$ for every $z \in A$.  Any
gated set is convex.  Moreover the set of gated sets of a graph with the addition of the empty set is a
convexity, and thus is closed under any intersections.   We will say that a
subgraph of a graph $G$ is \emph{gated} if its vertex set is gated.

\begin{lem}\label{L:gat.subgr./Uab}
Let $G$ be a partial cube, $F$ a gated subgraph of $G$, and $ab$ an edge of $F$.  Then the gate in $F$ of any $x \in U_{ab}^G$ belongs to $U_{ab}^F$.
\end{lem}

\begin{proof}
This is trivial if $x \in V(F)$.  Assume that $x \in V(G-F)$, and let $y$ be the neighbor of $x$ in $U_{ba}^G$.  Clearly, by Lemma~\ref{L:prop.isom subgr.},
\begin{gather*}
W_{ab}^F \subseteq W_{ab}^G \textnormal{\; and}\; W_{ba}^F \subseteq W_{ba}^G\\
U_{ab}^F \subseteq U_{ab}^G \textnormal{\; and}\; U_{ba}^F \subseteq U_{ba}^G
\end{gather*}
since $F$ is convex in $G$.

Denote by $g(x)$ and $g(y)$ the gates in $F$ of $x$ and $y$, respectively.  Then $g(x) \in I_G(x,a)$ and $g(y) \in I_G(y,b)$.  Hence $g(x) \in W_{ab}^F$ and $g(y) \in W_{ba}^F$.  On the other hand $y, g(x) \in I_G(x,g(y))$ and $x, g(y) \in I_G(y,g(x))$.  It easily follows that the vertices $g(x)$ and $g(y)$ are adjacent.  Therefore $g(x) \in U_{ab}^F$ and $g(y) \in U_{ba}^F$.
\end{proof}

\begin{thm}\label{T:gat.subg./ph}
Let $F$ be a gated subgraph of a partial cube $G$ such that $ph(G) \leq 1$.  Then $ph(F) \leq 1$.
\end{thm}

\begin{proof}
Let $ab$ be an edge of $F$.  By Lemma~\ref{L:prop.isom subgr.}, we have $U_{ab}^F \subseteq U_{ab}^G$ and $U_{ba}^F \subseteq U_{ba}^G$ since $F$ is convex in $G$.  We will show that $U_{ab}^{F}$ is ph-stable.

Let $x, y \in \mathcal{I}_{F}(U_{ab}^{F})$.  Because $\mathcal{I}_{F}(U_{ab}^{F}) \subseteq \mathcal{I}_{G}(U_{ab}^{G})$, and since $U_{ab}^G$ is ph-stable by Proposition~\ref{P:charact.p.c.ph}, it follows that $y \in I_G(x,z)$ for some $z \in U_{ab}^G$.  By Lemma~\ref{L:gat.subgr./Uab}, the gate $g(z)$ of $z$ in $F$ belongs to $U_{ab}^F$.  Moreover $y \in I_F(x,g(z))$ since $g(z) \in I_G(y,z)$.  Consequently $U_{ab}^F$ is ph-stable.  

In the same way we can prove that $U_{ba}^F$ is ph-stable.  We infer that $ph(F) \leq 1$ from Proposition~\ref{P:charact.p.c.ph}.
\end{proof}

We recall that a graph $G$ is the \emph{gated amalgam} of two graphs
$G_{0}$ and $G_{1}$ if $G_{0}$ and $G_{1}$ are isomorphic to two
intersecting gated subgraphs $G'_{0}$ and $G'_{1}$ of $G$ whose union
is $G$.  More precisely we also say that $G$ is the gated amalgam of
$G_{0}$ and $G_{1}$ \emph{along} $G'_{0} \cap G'_{1}$.  The gated amalgam of two partial cubes is clearly a partial cube.

\begin{thm}\label{T:gat.amalgam}
Let $G$ be the gated amalgam of two partial cubes $G_0$ and $G_1$.  Then $ph(G) \leq 1$ if and only if $ph(G_i) \leq 1$ for $i = 0, 1$.
\end{thm}

\begin{proof}
The necessity is clear by Theorem~\ref{T:gat.subg./ph} since $G_0$ and $G_1$ are isomorphic to two gated subgraphs of $G$.  Conversely, assume that $G = G_{0} \cup G_{1}$ where, for $i = 0, 1$,\; $G_{i}$ is a gated subgraph of $G$ such that $ph(G_i) \leq 1$.  The subgraph $G_{01} := G_{0} \cap G_{1}$ is also gated in $G$ as an intersection of gated subgraphs.  Let $ab$ be an edge of $G$.  We will show that $U_{ab}^G$ is ph-stable.  We distinguish two cases.

\emph{Case 1.}\; $U_{ab}^{G} = U_{ab}^{G_{i}}$ for some $i = 0$ or $1$.

Then $U_{ab}^G$ is ph-stable since so is $U_{ab}^{G_{i}}$ by Proposition~\ref{P:charact.p.c.ph}.

\emph{Case 2.}\; $U_{ab}^{G} \neq U_{ab}^{G_{i}}$ for $i = 0, 1$.

Then, for $i = 0, 1$, $G_{i}$ has an edge which is $\Theta$-equivalent
to $ab$.  Hence $G_{01}$, which is gated in $G$, also has an edge
$\Theta$-equivalent to $ab$.  Then, without loss of generality we can
suppose that $ab \in E(G_{01})$.  For any $x \in V(G)$ and $i = 0, 1$,
we denote by $g_{i}(x)$ the gate of $x$ in $G_{i}$.  Clearly 
\begin{gather}
\label{E:Wab}
W_{ab}^{G} =
W_{ab}^{G_{0}} \cup W_{ab}^{G_{1}} \textnormal{\quad and}\quad W_{ba}^{G} =
W_{ba}^{G_{0}} \cup W_{ba}^{G_{1}}\\
\label{E:Uab}
U_{ab}^{G} =
U_{ab}^{G_{0}} \cup U_{ab}^{G_{1}} \textnormal{\quad and}\quad U_{ba}^{G} = U_{ba}^{G_{0}} \cup U_{ba}^{G_{1}}\\
 \label{E:Iab}
 \mathcal{I}_{G_0}(U_{ab}^{G_0}) \cup \mathcal{I}_{G_1}(U_{ab}^{G_1}) \subseteq \mathcal{I}_{G}(U_{ab}^{G}).
\end{gather}

Let $u, v \in \mathcal{I}_{G}(U_{ab}^{G})$.  If $u, v \in \mathcal{I}_{G}(U_{ab}^{G_i})$ for some $i = 0$ or $1$, then $v \in I_{G_i}(u,w)$ for some $w \in U_{G_i}(ab)$.  Hence we are done because $v \in I_{G}(u,w)$ by (\ref{E:Iab}) and $w \in U_{G}(ab)$ by (\ref{E:Uab}).

Suppose that $u \in V(G_0)-V(G_1)$ and $v \in V(G_1)-V(G_0)$.  We first show that $u \in \mathcal{I}_{G_0}(U_{ab}^{G_0})$.  Because $u \in V(G_0)-V(G_1)$, we can suppose that $u \in I_G(x,y)$ for some $x \in U_{ab}^{G_{0}}$ and $y \in U_{ab}^{G_{1}}$.  Then $g_0(y) \in U_{ab}^{G_{0}}$ by Lemma~\ref{L:gat.subgr./Uab}, and thus $u \in I_{G_0}(x,g_0(y))$ since $g_0(y) \in I_{G_0}(u,y)$.  It follows that $g_1(u) \in I_{G_{01}}(g_1(x),g_0(y)) \subseteq \mathcal{I}_{G_1}(U_{ab}^{G_1})$.  Analogously $v \in \mathcal{I}_{G_1}(U_{ab}^{G_1})$.  Hence $v \in I_{G_1}(g_1(u),w)$ for some $w \in U_{ab}^{G_1}$ because $U_{ab}^{G_1}$ is ph-stable by Proposition~\ref{P:charact.p.c.ph}.  We infer that $v \in I_G(u,w)$, which proves that $U_{ab}^G$ is ph-stable.

In the same way we can prove that $U_{ba}^G$ is ph-stable.  Consequently $ph(G) \leq 1$ by Proposition~\ref{P:charact.p.c.ph}.
\end{proof}

We recall below three well-known properties of the
cartesian product that we will use in the proof of the next theorem.  The cartesian product of two partial cubes is clearly a partial cube.

\begin{pro}\label{P:pro.Cartes.prod.}
Let $G = G_{0} \Box G_{1}$ be the cartesian product of two connected 
graphs.  We have the following properties:

\textnormal{Distance Property:}\; $d_{G}(x,y) = d_{G_{0}}(pr_{0}(x),pr_{0}(y)) + 
d_{G_{1}}(pr_{1}(x),pr_{1}(y))$ for any $x, y \in V(G)$.

\textnormal{Interval Property:}\; $I_{G}(x,y) = I_{G_{0}}(pr_{0}(x),pr_{0}(y))
\times I_{G_{1}}(pr_{1}(x),pr_{1}(y))$ for any $x, y \in V(G)$.

\textnormal{Convex Subgraph Property:}\; A subgraph $F$ of $G$ is convex if and 
only if $F = pr_{0}(F) \Box pr_{1}(F)$, where both $pr_{0}(F)$ and 
$pr_{1}(F)$ are convex.
\end{pro}

\begin{thm}\label{T:cart.prod.}
Let $G = G_0 \Box G_1$ be the cartesian product of two partial cubes $G_0$ and $G_1$.  Then $ph(G) \leq 1$ if and only if $ph(G_i) \leq 1$ for $i = 0, 1$.
\end{thm}

\begin{proof}
Assume that $ph(G) \leq 1$.  Let $F_i$ be a $G_i$-fiber of $G$ for some $i = 0$ or $1$.  Then $F_i$ is a gated subgraph of $G$.  Indeed, by the Distance Property of the cartesian product, the projection onto $F_i$ of any vertex $x$ of $G$ is the gate of $x$ in $F_i$.  Therefore, by Theorem~\ref{T:gat.subg./ph}, $F_i$, and thus $G_i$, has a pre-hull number which is at most $1$.

Conversely, assume that $ph(G_i) \leq 1$ for $i = 0, 1$.  For any $x \in V(G)$, we denote by $x_0$ and $x_1$ the projections of $x$ onto $G_0$ and $G_1$, respectively, i.e. $x = (x_0,x_1)$.  Let $ab \in E(G)$.  Then $a_i = b_i$ for exactly one $i$, say $i = 1$.  We will show that $U_{ab}^G$ is ph-stable.

Clearly, for any $cd$ of $G$ is $\Theta$-equivalent to $ab$ if and only if $c_1 = d_1$ and $c_0d_0$ is $\Theta$-equivalent to $a_0b_0$.  Hence
\begin{equation}
\label{E:U'ab}
U_{ab}^G = U_{a_0b_0}^{G_0} \times V(G_1).
\end{equation}

Let $u, v \in \mathcal{I}_G(U_{ab}^G)$.  By the Interval Property of the cartesian product, $u_0, v_0 \in \mathcal{I}_{G_0}(U_{a_0b_0}^{G_0})$.  Then, because $U_{a_0b_0}^{G_0}$ is ph-stable by Proposition~\ref{P:charact.p.c.ph}, it follows that $v_0 \in I_{G_0}(u_0,w_0)$ for some $w_0 \in U_{a_0b_0}^{G_0}$.  In the case where $u_0 = v_0$, we can choose $w_0$ as any element of $U_{a_0b_0}^{G_0}$.  Let $w := (w_0,v_1)$.  Then $w \in U_{ab}^G$ by (\ref{E:U'ab}), and $v \in I_G(u,w)$ by the Distance Property of the cartesian product.  This proves that $U_{ab}^G$ is ph-stable.

In the same way we can prove that $U_{ba}^G$ is ph-stable.  Consequently $ph(G) \leq 1$ by Proposition~\ref{P:charact.p.c.ph}.
\end{proof}

From the above theorems we infer the following result:

\begin{cor}\label{C:closure}
The class of all partial cubes whose
pre-hull number is at most~$1$ is closed under gated  subgraphs, gated amalgams and cartesian products.
\end{cor}

\end{document}